\documentclass[letterpaper,11pt]{amsart}

\usepackage[all]{xy}
\usepackage{amsfonts}
\usepackage{amsmath}
\usepackage{amssymb}
\usepackage{amsthm}

\newtheorem{thm}{Theorem}[section]

\newtheorem{quest}{Open Question}
\newtheorem{cor}[thm]{Corollary}
\newtheorem{lem}[thm]{Lemma}
\newtheorem{prop}[thm]{Proposition}

\theoremstyle{definition}
\newtheorem{defn}[thm]{Definition}

\newtheorem{exam}[thm]{Example}

\theoremstyle{remark}
\newtheorem{rem}[thm]{Remark}

\numberwithin{equation}{section}

\begin{document}

\author{gangyong lee}
\title{The rational hull of modules}

\address{}
\address{Gangyong Lee, Department of Mathematics Education, Chungnam National University}
\address{Daejeon 34134, Republic of Korea}
\address{e-mail: lgy999$@$cnu.ac.kr}


\begin{abstract} In this paper, we provide several new characterizations of the maximal right ring of quotients of a ring by using the relatively dense property. As a ring is embedded in its maximal right ring of quotients, we show that the endomorphism ring of a module is embedded into that of the rational hull of the module. In particular, we obtain new characterizations of rationally complete modules.
The equivalent condition for the rational hull of the direct sum of modules to be the direct sum of the rational hulls of those modules under certain assumption is presented. 
For a right $H$-module $M$ where $H$ is a right ring of quotients of a ring $R$, we provide a sufficient condition to be $\text{End}_R(M)=\text{End}_H(M)$. Also, we give a condition for the maximal right ring of quotients of the endomorphism ring of a module to be the
endomorphism ring of the rational hull of a module.
\end{abstract}

\maketitle


\section{Introduction}
The theory of rings of quotients has its origin in the work of \O. Ore \cite{oo} and K. Asano \cite{as} on the construction of the total ring of fractions, in the 1930's and 40's. But the subject did not really develop until the end of the 1950's, when a number of important papers appeared (by R.E. Johnson \cite{rej}, Y. Utumi \cite{ut2}, A.W. Goldie \cite{awg}, J. Lambek \cite{lambek1} et al). In particular, Johnson(1951), Utumi(1956), and Findlay $\&$ Lambek(1958) have studied the maximal right ring of quotients of a ring which is an extended ring of the base ring. For the well-known example, the maximal right ring of quotients of integers is rational numbers. It is the same as the injective hull of integers. For a commutative ring, the classical right ring of quotients of a ring is to its total quotient ring as the maximal right ring of quotients of a ring is to its complete ring of quotients.

As we know, the study of the rational hull of a module is the same as that of the maximal right ring of quotients in a different way.
Also, like every module has the injective hull, it is known that every module has the rational hull in \cite[Theorem 2.6]{fl}.
Now, we introduce the definition of the rational hull of a module and present its well-known results, briefly. Let $M$ be a right $R$-module and $T=\text{End}_R(E(M))$. Put\\
$\widetilde E(M)=\{x \in E(M)| \,\vartheta(M)=0 ~\text{with}~  \vartheta\in T ~ \Rightarrow \ \vartheta(x)=0\}=\displaystyle\bigcap_{\substack{M\subseteq\text{Ker}\vartheta \\ \vartheta\in T}}\text{Ker}\vartheta=\mathbf{r}_{E(M)}\left(\mathbf{l}_T(M)\right).$
Then $\widetilde E(M)$ is the unique maximal rational extension of $M$. We call it the \emph{rational hull} of $M$.
Also, it is known that $\mathbf{r}_{E(M)}(J(T)) \leq\mathbf{r}_{E(M)}\left(\mathbf{l}_T(M)\right)=\widetilde E(M)$ because $\mathbf{l}_T(M)\subseteq J(T)$ where $J(T) = \{ \alpha\in T\, |\, \text{Ker}\alpha\leq^\text{ess} E(M)\}$ is a Jacobson radical of a ring $T$.
Note that the maximal right ring of quotients of $R$ is $Q(R)=\mathbf{r}_{E(R)}(\mathbf{l}_H(R))$ where $H=\text{End}_R(E(R))$ (see \cite[Proposition 2]{lambek1}). 

After the necessary background history, results, and notations in this section, we provide several characterizations of the rational hull of a module in Section 2 (see Theorem \ref{char} and Corollary \ref{char2}).
 In addition, characterizations of rationally complete modules are presented. As a corollary, we obtain several new characterizations of the maximal right ring of quotients of a ring. In particular, we show that the endomorphism ring of a module is embedded into that of the rational hull of the module as the inherited property of its maximal right ring of quotients (see Theorem \ref{relationship}). 
Our focus, in Section 3, is on the question of when is the rational hull of the direct sum of modules the direct sum of the rational hulls of those modules.
For $M=\bigoplus_{k\in \Lambda} M_k$, we prove that $\widetilde{E}(M)=\bigoplus_{k\in \Lambda} \widetilde{E}(M_k)$ if and only if $M_i$ is $M_j$-{dense} in $\widetilde{E}(M_i)$ for all $i, j \in \Lambda$ when either $R$ is right noetherian or $|\Lambda|$ is finite (see Theorem \ref{directsum}). 
In the last section, we obtain a condition to be $\text{End}_R(M)=\text{End}_H(M)$ where $H$ is a right ring of quotients of a ring $R$ (Theorem \ref{subprojectivemaximal}). This condition is called the \emph{relatively dense property} to a module. Also, we provide a sufficient condition for the maximal right ring of quotients of the endomorphism ring of a module to be the endomorphism ring of the rational hull of a module (see Theorem \ref{maximalrational}).

\vspace{0.2cm}

Throughout this paper, $R$ is a ring with unity and $M$ is a unital
right $R$-module. For a right $R$-module $M$, $S=\text{End}_R(M)$
denotes the endomorphism ring of $M$; thus $M$ can be viewed as a
left $S$- right $R$-bimodule. For $\varphi \in S$, $\text{Ker}
\varphi$ and $\text{Im}\varphi$ stand for the kernel and the image
of $\varphi$, respectively. The notations $N\leq M$, $N\leq^\text{ess} M$, $N\leq^\text{den} M$ 
or $N\leq^\oplus M$ mean that $N$ is a
submodule, an essential submodule, a dense submodule, or a
direct summand of $M$, respectively.
By $E(M)$, $\widehat M$, and $\widetilde E(M)$ we denote the injective hull, the quasi-injective hull, and the rational hull  of $M$, respectively, and $T=\text{End}_{R}(E(M))$. $Q(R)$ denotes the maximal right ring of quotients of $R$. 
 The direct sum of $\Lambda$ copies of
$M$ is denoted by $M^{(\Lambda)}$ where $\Lambda$ is an arbitrary index set. $\mathsf{CFM}_{\mathbb{N}}(F)$ denotes
the $\mathbb{N}\times\mathbb{N}$ column finite matrix ring over a field $F$.
By $\mathbb{Q}$, $\mathbb{Z}$, and $\mathbb{N}$ we denote the set of
rational, integer, and natural numbers, respectively.
$\mathbb{Z}_n$ denotes the $\mathbb{Z}$-module
$\mathbb{Z}/n\mathbb{Z}$. For $x\in M$, $x^{-1}K=\{r\in R\,|\,xr\in K\}\leq R_R$ with a right $R$-submodule $K$ of $M$.
We also denote $\mathbf{r}_M(I)=\{m\in M\,|\, Im=0\}$ for
$I\leq S$ and $\mathbf{l}_S(N)=\{\varphi\in S\,|\, \varphi N=0\}$ for
$N\leq M$. 

\vspace{0.2cm}

We give some properties of dense submodules.
Recall that a submodule $N$ of $M$ is said to be \emph{dense} in $M$ if for any $x, 0\neq y\in M,$ there exists $r\in R$ such that $xr\in N$ and $0\neq yr$.
\begin{prop}[{\cite[Proposition 1.3.6]{bprbook}}]\label{pro138} Let $N\leq M$ be right $R$-modules. Then the following conditions are equivalent:
\begin{enumerate}
\item [(a)] $N$ is dense in $M$;
\item [(b)] $\emph{Hom}_R(M/N, E(M))=0$;
\item [(c)] for any submodule $P$ such that $N\leq P\leq M$, $\emph{Hom}_R(P/N, M)=0$.
\end{enumerate}
\end{prop}

\begin{prop}[{\cite[Proposition 8.7]{l}}]\label{pro87} Let $L, N$ be submodules of a right $R$-module $M$:
\begin{enumerate}
\item [(i)] If $L\leq^\emph{den}M$ and $N\leq^\emph{den}M$ then  $L\cap N\leq^\emph{den}M$.
\item [(ii)] Let $L\leq V \leq M$. Then $L\leq^\emph{den}M$ if and only if $L\leq^\emph{den}V$ and $V\leq^\emph{den}M$.
\end{enumerate}
\end{prop}
\begin{prop}[{\cite[Proposition 1.3.7]{bprbook}}]\label{pro137} Let $M$ be a right $R$-module and
$M\leq V\leq{E}(M)$. Then $M\leq^\emph{den}V$ if and only if $V\leq\widetilde{E}(M)$.
\end{prop}

We remind of some important characterizations of the rational hull of a module. 
\begin{prop}\label{rationalhull} The following hold true for a right $R$-module $M$ and $T=\emph{End}_R(E(M))$:
\begin{enumerate}
\item [(i)]\emph{({\cite[Exercises 5]{lambek}})} $\widetilde E(M)=\{x \in E(M)|\,~ \vartheta|_M=1_M ~\emph{with} ~\vartheta\in T\ ~ \Rightarrow \  \vartheta(x)=x\}$.
\item [(ii)]\emph{({\cite[Proposition 8.16]{l}})} $\widetilde E(M)=\{x\in E(M)\,|\,~ \forall y\in E(M)\! \setminus\!\{0\},~ y\cdot x^{-1}M\neq0\}$.
\end{enumerate}
\end{prop}
\vspace{0.2cm}

\section{The rational hull of a module}

As the injective hull of a module $M$ is the minimal injective module including $M$, the next result shows that the rational hull of a module $M$ is the minimal rationally complete module including $M$. 
Recall that a right $R$-module $M$ is said to be \emph{rationally complete} if it has no proper rational (or dense) extensions, or equivalently $\widetilde{E}(M)=M$.
Thus, the rational hull $\widetilde{E}(M)$ of a module $M$ is rationally complete. 

\begin{thm}\label{racom} The following conditions are equivalent for right $R$-modules $M$ and $F$:
\begin{enumerate}
\item [(a)] $F$ is maximal dense over $M$;
\item [(b)] $F$ is rationally complete, and is dense over $M$;
\item [(c)] $F$ is minimal rationally complete, and is essential over $M$.
\end{enumerate}
Note that a right $R$-module $F$ is exactly the {rational hull} of a module $M$ if $F$ satisfies any one of the above equivalent conditions.
\end{thm}
\begin{proof} (a)$\Rightarrow$(b) From Proposition \ref{pro137}, it is easy to see that $F$ has no proper dense extension. So, $F$ is a rationally complete module.
(b)$\Rightarrow$(c) Let $F'$ be rationally complete such that $M\leq F'\leq F$.
Since $M\leq^\text{den}F$,  from Proposition \ref{pro87}(ii) $M\leq^\text{den} F'\leq^\text{den}F$.
Thus, from Proposition \ref{pro137} $F\leq^\text{den}\widetilde{E}(F')=F'$ because $F'$ is rationally complete.
Therefore $F=F'$.
(c)$\Rightarrow$(a) Let $F$ be minimal rationally complete over $M$.
Since $F$ is essential over $M$, $M\leq F\leq E(M)$.
Since $M\leq^\text{den}\widetilde{E}(M)$, $\text{Hom}_R(\widetilde{E}(M)/M, E(M))=0$.
Also, since $E(F)=E(M)$, $\text{Hom}_R(\widetilde{E}(M)/M, E(F))=0$.
From \cite[Theorem 8.24]{l}, an inclusion map $\iota: M\rightarrow F$ extends to $\rho: \widetilde{E}(M)\rightarrow F$ as $F$ is rationally complete (see also Proposition \ref{relationship3}).
Note that $\rho$ is a monomorphism. Since $\widetilde{E}(M)$ is rationally complete and $F$ is minimal, $\widetilde{E}(M)=F$.
\end{proof}

The next example shows that the condition ``essential over $M$" in Theorem \ref{racom}(c) is not superfluous.
 \begin{exam}
Let $M = \mathbb{Z}$ and $F=\mathbb{Z}_{(p)}\oplus\mathbb{Z}_p$ be right $\mathbb{Z}$-modules where $\mathbb{Z}_{(p)}$ is the localization of $\mathbb{Z}$ at the prime ideal $(p)$.
It is easy to see that $M$ is not essential in $F$, so $F$ is not a rational hull of $M$.
In fact, $F$ is minimal rationally complete over $M$: 
From \cite[Example 8.21]{l}, $F$ is rationally complete because  $F$ is the rational hull of $L=\mathbb{Z} \oplus \mathbb{Z}_p$.
It is enough to show that $F$ is minimal over $M$:  Let $K$ be a rationally complete module such that $M \leq K \leq F$. Hence $1=\text{u.dim}(M) \le \text{u.dim}(K) \le \text{u.dim}(F)=2$. Assume that $\text{u.dim}(K)=1.$ Then $M \leq^\text{ess} K$, and hence $K$ is nonsingular since $M$ is nonsingular. Thus $M \le^\text{den} K$, which implies that $K \cong \mathbb{Q}$ since $K$ is rationally complete and $\widetilde{E}(M)=\mathbb{Q}$. It follows that $\mathbb{Q}$ can be embedded into $F=\mathbb{Z}_{(p)} \oplus \mathbb{Z}_p$, a contradiction. Therefore, $\text{u.dim}(K)=2.$ Then $K \leq^\text{ess} F$, and hence $K \cap \mathbb{Z}_p \neq 0$. Thus $\mathbb{Z}_p \le K$, which implies that $L=\mathbb{Z} \oplus \mathbb{Z}_p \le K$. Note that $L \le^\text{den} F$ since $F=\widetilde{E}(L)$. Hence $K \le^\text{den} F$, so that $K=F$ due to the fact that $K$ is rationally complete.
\end{exam}

We provide another characterization for the rational hull of a module using the relatively dense property.
A right ideal $I$ of a ring $R$ is called \emph{relatively dense to a right $R$-module} $M$ (or $M$-\emph{dense}) in  $R$ if for any $r\in R$ and $0\neq m\in M$, $m\cdot r^{-1}I\neq 0$. It is denoted by $I\leq^\text{den}_M R$.

\begin{thm}\label{char}
	Let $M$ be a right $R$-module. Then $\widetilde{E}(M)=\{x \in E(M)\,|\,x^{-1}M \le^{\emph{den}}_M R\}$.
\end{thm}
\begin{proof}
	Let $x \in \widetilde{E}(M)$ be arbitrary. Consider a right ideal $x^{-1}M \leq R$.
Let	$0 \neq m \in M$ and $r \in R$. Since $M \leq^{\text{den}} \widetilde{E}(M)$, there exists $s \in R$ such that $ms \neq 0$ and $(xr)s=x(rs) \in M$, that is, $rs\in x^{-1}M$. Hence $x^{-1}M \leq^{\text{den}}_M R$. 
	
For the reverse inclusion, let $x \in E(M)$ such that $x^{-1}M \leq^{\text{den}}_M R$. For an arbitrary nonzero element $0 \neq y \in E(M)$, it suffices to show that $y\cdot x^{-1}M \neq 0$. 
	As $M \leq^{\text{ess}} E(M)$, $0\neq yr\in M$ for some  $r \in R$. 
	Since $x^{-1}M \leq^{\text{den}}_M R$, there exists $s \in R$ such that $yrs \neq 0$ and $rs \in x^{-1}M$. Hence $0 \neq yrs \in y\cdot x^{-1}M$. Therefore  $x \in \widetilde{E}(M)$.
\end{proof}

 The next definition was shown in \cite[pp79]{fl} as $N\leq M (K)$, so we call a submodule $N$ relatively dense to a module $K$ in a module $M$. (For details, see \cite{ltz1}.)

\begin{defn}\label{redense} A submodule $N$ of a right $R$-module $M$ is said to be \emph{relatively dense to a right $R$-module} $K$ (or $K$-\emph{dense}) in $M$ if for any $m\in M$ and $0\neq x\in K$, $x\cdot m^{-1}N\neq 0$, denoted by $N \leq^\text{den}_K M$. Note that $N$ is $M$-dense in $M$ if and only if $N$ is dense in $M$. 
\end{defn}

 We provide some characterizations of the relative density property. One can compare the following characterizations to Proposition \ref{pro138}. The equivalence (a)$\Leftrightarrow$(c) in the following proposition is provided by \cite[pp79]{fl}.

\begin{prop}\label{ndense} The following are equivalent for right $R$-modules $M, K$ and $N\leq M$:
\begin{enumerate}
\item [(a)] $N$ is $K$-dense in $M$;
\item [(b)] $\emph{Hom}_R(M/N, E(K))=0$;
\item [(c)] for any submodule $P$ such that $N\leq P\leq M$, $\emph{Hom}_R(P/N, K)=0$.
\end{enumerate}
\end{prop}
\begin{proof} (a)$\Rightarrow$(b) Assume that there exists $0\neq\alpha \in\text{Hom}_R(M, E(K))$ with $\alpha N=0$.
Since $\alpha M\cap K\neq 0$, there exist $x\in M$ and $0\neq y\in K$ such that $\alpha(x)=y$.
Since $N$ is $K$-{dense} in $M$, there exists $r\in R$ such that $xr\in N$ and $0\neq yr$.
However, $0=\alpha(xr)=\alpha(x)r=yr\neq0$, a contradiction. Hence $\text{Hom}_R(M/N, E(K))=0$.

(b)$\Rightarrow$(c) Assume that for any submodule $P$ such that $N\leq P\leq M$, there exists $0\neq\eta\in\text{Hom}_R(P/N, K)$.
Then by the injectivity of $E(K)$, we can extend $\eta$ to a nonzero homomorphism from $M/N$ to $E(K)$, a contradiction.
Hence $\text{Hom}_R(P/N, K)=0$.

(c)$\Rightarrow$(a) Assume that $y\cdot x^{-1}N=0$ for some $x\in M$ and $0\neq y\in K$.
We define $\gamma: N+xR \rightarrow K$ given by $\gamma(n+xr)=yr$ for $n\in N$ and $r\in R$.
It is easy to see that $\gamma$ is a well-defined $R$-homomorphism vanishing on $N$.
Since $N\leq N+xR\leq M$, by hypothesis $0=\gamma(x)=y\neq0$, a contradiction. Thus $N$ is $K$-dense in $M$.
\end{proof}

We obtain another characterization of the relative density property related to homomorphisms.
\begin{prop}\label{rmdr2} Let $M, K$ be right $R$-modules. Then a submodule $N$ is $K$-dense in $M$ if and only if $\mathbf{l}_{H}(N)=0$ where $H=\emph{Hom}_R(M, E(K))$.
\end{prop}
\begin{proof}
Suppose $N$ is $K$-dense in $M$. Assume that $0\neq \varphi\in H$ such that $\varphi N=0$.
Then there exists $m\in M\setminus N$ such that $\varphi (m)\neq 0$. 
Since $\varphi (m)\in E(K)$, $0\neq \varphi (m)r\in K$ for some $r\in R$.
Hence  there exists $s\in R$ such that $mrs\in N$ and $\varphi (m)rs\neq0$  because $N\leq^\text{den}_K M$. That yields a contradiction that $0\neq \varphi (m)rs=\varphi (mrs)\in \varphi N=0$. Therefore $\mathbf{l}_{H}(N)=0$.
Conversely, assume that $x \cdot m^{-1}N = 0$ for some $0\neq x\in K$ and $m\in M$.
We define $\gamma : N+mR\rightarrow E(K)$ by $\gamma(n+mt)=xt$ for $n\in N$ and $t\in R$.
Clearly, $\gamma$ is a nonzero $R$-homomorphism vanishing on $N$.
Also, there exists $\overline{\gamma}: M\rightarrow E(K)$ such that $\overline{\gamma}|_{N+mR}=\gamma$.
Since $0=\overline{\gamma}N$, $0\neq\overline{\gamma}\in\mathbf{l}_{H}(N)$, a contradiction. Therefore $x \cdot m^{-1}N \neq 0$.
\end{proof}

If $M=R$, the following result is directly provided.
\begin{cor}[{\cite[Proposition 1.1]{sto}}]\label{rmdr} Let $K$ be a right $R$-module and $I$ be a right ideal of a ring $R$. Then $I$ is $K$-dense in  $R$ if and only if $\mathbf{l}_{E(K)}(I)=0$.
\end{cor}

\begin{prop}\label{rmdr3} Let $K$ be a right $R$-module and $I$ be an ideal of a ring $R$. Then $\mathbf{l}_{K}(I)=0$ if and only if $\mathbf{l}_{E(K)}(I)=0$.
\end{prop}
\begin{proof} Since one direction is trivial, we need to show the other direction. Suppose $\mathbf{l}_{K}(I)=0$.
Assume that $\mathbf{l}_{E(K)}(I)\neq 0$. Then there exists $0\neq x\in E(K)$ such that $xI=0$. Also, $0\neq xr\in K$ for some $r\in R$  because $K\leq^\text{ess}E(K)$. Since $xrI\subseteq xI=0$, $0\neq xr\in\mathbf{l}_{K}(I)$, a contradiction.
Therefore $\mathbf{l}_{E(K)}(I)=0$.
\end{proof}

\begin{cor}[{\cite[Proposition 1.3.11(iv)]{bprbook}}]\label{rmdr4} Let $I$ be an ideal of a ring $R$. Then $I\leq^\emph{den} R_R$ if and only if $\mathbf{l}_{R}(I)=0$.
\end{cor}
\begin{proof} The proof also follows from Corollary \ref{rmdr} and Proposition \ref{rmdr3}.
\end{proof}

Using Theorem \ref{char} and Corollary \ref{rmdr}, we obtain another characterization for the rational hull of a module. Also, using the characterization of the relatively dense property, new characterization for the rational hull of a module is provided.
\begin{cor}\label{char2} Let $M$ be a right $R$-module. Then the following statements hold true:
\begin{enumerate}
\item [(i)] \emph{({\cite[Proposition 1.4(b)]{sto}})} $\widetilde{E}(M)=\{x \in E(M)\,|\,\mathbf{l}_{E(M)}(x^{-1}M)=0\}$.
\item [(ii)]  $\widetilde{E}(M) =\{x\in E(M)\,|\,\emph{Hom}_R(R/x^{-1}M, E(M))=0\}$.
\end{enumerate}
\end{cor}
\begin{proof}
It directly follows from Theorem \ref{char}, Corollary \ref{rmdr}, and Proposition \ref{ndense}.
\end{proof}

Several new characterizations for the maximal right ring of quotients of a ring are provided as the following. 
\begin{thm} Let $R$ be a ring. Then the following statements hold true:
\begin{enumerate}
\item [(i)] A right ideal $I$ is dense in $R$ if and only if $\mathbf{l}_{E(R)}(I)=0$.
\item [(ii)] $Q(R)=\{x \in E(R)\,|\,x^{-1}R \leq^{\emph{den}} R\}$.
\item [(iii)] $Q(R)=\{x \in E(R)\,|\,\mathbf{l}_{E(R)}(x^{-1}R)=0\}$.
\item [(iv)] $Q(R)=\{x\in E(R)\,|\,\emph{Hom}_R(R/x^{-1}R, E(R))=0\}$.
\end{enumerate}
\end{thm}

We give characterizations for a rationally complete module.
\begin{thm}\label{rationally complete} The following conditions are equivalent for a right $R$-module $M$:
\begin{enumerate}
\item[(a)] $M$ is a rationally complete module;
\item[(b)] $\{\overline{x}\in E(M)/M\,|\, \mathbf{l}_{\text{E}(\text {M})}(\mathbf{r}_\text{R}(\overline{x}))=0\}=\overline{0}$;
\item[(c)] For any $I\leq^\emph{den}_MR$, $\varphi\in \emph{Hom}_R(I, M)$ can be uniquely extended to $\widetilde{\varphi}\in \emph{Hom}_R(R, M)$.
\end{enumerate}
 \end{thm}
\begin{proof} Take $A:=\{\overline{x}\in E(M)/M\,|\, \mathbf{l}_{\emph{E}(\emph {M})}(\mathbf{r}_\text{R}(\overline{x}))=0\}$. (a)$\Rightarrow$(b)  Assume that $x\in E(M)\setminus M$ such that $\overline{x}\in A$. From Corollary \ref{rmdr}, $\mathbf{r}_\text{R}(\overline{x})\leq^\text{den}_M R$. Since $\mathbf{r}_\text{R}(\overline{x})=x^{-1}M$, $x^{-1}M\leq^\text{den}_M R$. Hence from Theorem \ref{char} $x\in \widetilde E(M)=M$ because $M$ is rationally complete, a contradiction. Therefore $A=\overline{0}$. (b)$\Rightarrow$(c) Assume to the contrary of the condition (c). For $I\leq^\text{den}_MR$, since $M \subseteq E(M)$, there exists $\varphi\in \text{Hom}_R(I, M)$ such that $\widetilde{\varphi}\in \text{Hom}_R(R, E(M))$, $\widetilde{\varphi}|_I=\varphi$, and $\widetilde{\varphi}(1)\notin M$. Since $\overline{0}\neq \widetilde{\varphi}(1)+M\in E(M)/M$ and $I\subseteq\mathbf{r}_\text{R}(\widetilde{\varphi}(1)+M)\leq^\text{den}_M R$, $\mathbf{l}_{\emph{E}(\emph {M})}(\mathbf{r}_\text{R}(\widetilde{\varphi}(1)+M))=0$ from Corollary \ref{rmdr}, a contradiction that $A=\overline{0}$. Therefore $\varphi\in \text{Hom}_R(I, M)$ is extended to $\widetilde{\varphi}\in \text{Hom}_R(R, M)$. For the uniqueness, the proof is similar to that of Proposition \ref{relationship3}. 
 (c)$\Rightarrow$(a) Assume that $M$ is not rationally complete. Then there exists $x\in \widetilde E(M)\setminus M$ such that $x^{-1}M\leq^\text{den}_M R$ from Theorem \ref{char}. 
Define $\varphi: x^{-1}M \rightarrow M$ given by $\varphi(r)=xr$. By hypothesis, $\widetilde{\varphi}
(1)=x1=x\in M$, a contradiction. Therefore $M$ is rationally complete.
\end{proof}

Next, as a ring is embedding into its maximal right ring of quotients, we provide the relationship between the endomorphism rings of a module and its rational hull.
\begin{prop}\label{relationship3} Let $M$ and $K$ be right $R$-modules. For any $N\leq^\emph{den}_KM$, $\varphi\in \emph{Hom}_R(N, K)$ is uniquely extended to $\widetilde{\varphi}\in \emph{Hom}_R(M, \widetilde{E}(K))$ and $\widetilde{\varphi}|_N=\varphi$. In addition, $\varphi N\leq^\emph{den}\widetilde{\varphi} M$.
\end{prop}
\begin{proof}(Existence) Let $\varphi\in\text{Hom}_R(N, K)$ be arbitrary.
Then there exists $\widetilde{\varphi}\in \text{Hom}_R(M, E(K))$ such that
$\widetilde{\varphi}|_N=\varphi$.
Since $\widetilde{\varphi}$ induces a surjection from $M/N$ to $(\widetilde{\varphi}M+\widetilde{E}(K))/\widetilde{E}(K)$ and $\text{Hom}_R(M/N, E(K))=0$ (see Proposition \ref{ndense}), $\text{Hom}_R\left(\frac{\widetilde{\varphi}M+\widetilde{E}(K)}{\widetilde{E}(K)}, E(K)\right)=0$.
Hence $\widetilde{E}(K)\leq^\text{den}\widetilde{\varphi}M+\widetilde{E}(K)$ by Proposition \ref{pro138}.
As $\widetilde{E}(K)$ is rationally complete, $\widetilde{\varphi}M\subseteq \widetilde{E}(K)$.\\
\noindent (Uniqueness) Suppose $\widetilde{\varphi}$ and $\widetilde{\psi}$ are in $\text{Hom}_R(M, \widetilde{E}(K))$ such that $\widetilde{\varphi}|_N=\widetilde{\psi}|_N$.
It is enough to show that $\widetilde{\varphi}=\widetilde{\psi}$.
Assume that $\widetilde{\varphi}(x)\neq\widetilde{\psi}(x)$ for some $x\in M$.
Take $0\neq y=(\widetilde{\varphi}-\widetilde{\psi})(x)\in \widetilde{E}(K)$.
Thus, there exists $r\in R$ such that $0\neq yr\in K$.
Since $N\leq^\text{den}_KM$, there exists $s\in R$ such that $xrs\in N$ and $yrs\neq0$.
This yields a contradiction that $0\neq yrs=(\widetilde{\varphi}-\widetilde{\psi})(xrs)=(\widetilde{\varphi}|_N-\widetilde{\psi}|_N)(xrs)=0$.
Therefore $\widetilde{\varphi}=\widetilde{\psi}$.\\
In addition, let $x_1\in \widetilde{\varphi}M$ and $0\neq x_2\in \widetilde{\varphi}M$. Then $\widetilde{\varphi}(m_1)=x_1, \widetilde{\varphi}(m_2)=x_2$ for some $m_1, m_2\in M$. 
As $\widetilde{\varphi}M\subseteq \widetilde{E}(K)$, $0\neq x_2r\in K$ for some $r\in R$.
Since $N\leq^\text{den}_KM$ and $m_1r\in M$, there exists $s\in R$ such that $m_1rs\in N$ and $0\neq x_2rs$.
Thus $x_1rs=\widetilde{\varphi}(m_1rs)\in\varphi N$ and $0\neq x_2rs$.
Therefore $\varphi N\leq^\text{den}\widetilde{\varphi}M$.
\end{proof}
Note that the dense property implies the essential property, however  the relatively dense property does not imply the essential property in general: See $\mathbb{Z}_p\leq^\text{den}_\mathbb{Z} \mathbb{Z}_p\oplus\mathbb{Z}_p$ but $\mathbb{Z}_p\nleq^\text{ess} \mathbb{Z}_p\oplus\mathbb{Z}_p$ as a $\mathbb{Z}$-module. However, Proposition \ref{relationship3} shows that $\varphi N\leq^\text{den}\widetilde{\varphi} M$ when $N\leq^\text{den}_KM$ for any $\varphi\in \text{Hom}_R(N, K)$.
As a corollary, we have a generalized result of Theorem \ref{rationally complete}((a)$\Rightarrow$(b)).
\begin{cor} Let $M$ be a right $R$-module. If $K$ is rationally complete, then for any $N\leq^\emph{den}_KM$, $\varphi\in \emph{Hom}_R(N, K)$ is uniquely extended to $\widetilde{\varphi}\in \emph{Hom}_R(M, K)$ and $\widetilde{\varphi}|_N=\varphi$.
\end{cor}

\begin{thm}\label{relationship} Let $M$ be a right $R$-module.
Then $\emph{End}_R(M)$ is considered as a subring of $\emph{End}_R(\widetilde{E}(M))$.
\end{thm}
\begin{proof} Since $M\leq^\text{den} \widetilde{E}(M)$, from Proposition \ref{relationship3} $\varphi\in\text{End}_R(M)$ can be uniquely extended to  $\widetilde{\varphi}\in\text{End}_R(\widetilde{E}(M))$ because $\text{End}_R(M)\subseteq \text{Hom}_R(M, \widetilde{E}(M))$.
Thus we have a one-to-one correspondence between $\text{End}_R(M)$ and $\{\widetilde{\varphi}\in\text{End}_R(\widetilde{E}(M))\,|\,\widetilde{\varphi}|_M={\varphi}\in\text{End}_R(M)\}$ given by  $\Omega(\varphi)=\widetilde{\varphi}$.
We need to check that $\Omega$ is a ring homomorphism.

\noindent (i) Since $\Omega(\varphi+\psi)|_M=(\widetilde{\varphi+\psi})|_M=\varphi+\psi=\Omega(\varphi)|_M+\Omega(\psi)|_M=(\Omega(\varphi)+\Omega(\psi))|_M$,  from the uniqueness of Proposition \ref{relationship3} we have $\Omega(\varphi+\psi)= \Omega(\varphi)+\Omega(\psi)$.

\noindent (ii) Since $\Omega{(\varphi\circ\psi)}|_M=(\widetilde{\varphi\circ\psi})|_M=\varphi\circ\psi=\Omega(\varphi)|_M\circ\Omega(\psi)|_M=(\Omega(\varphi)\circ\Omega(\psi))|_M$ because $\Omega(\varphi)|_M\leq M$, from the uniqueness of Proposition \ref{relationship3} we have $\Omega(\varphi\circ\psi)= \Omega(\varphi)\circ\Omega(\psi)$.

\noindent Thus $\text{End}_R(M)$ is isomorphic to a subring of $\text{End}_R(\widetilde{E}(M))$.
Therefore we consider $\text{End}_R(M)$ as a subring of $\text{End}_R(\widetilde{E}(M))$.
\end{proof}

We conclude this section with results for the rational hulls of quasi-continuous modules and quasi-injective modules.
\begin{thm}\label{qiemqi} The following statements hold true for a module $M$:
\begin{enumerate}
\item[(i)] If $M$ is a quasi-continuous module then $\widetilde{E}(M)$ is a quasi-continuous module.
\item[(ii)] If $M$ is a quasi-injective module then $\widetilde{E}(M)$ is a quasi-injective module.
\end{enumerate}
 \end{thm}
\begin{proof} (i) Let $T=\text{End}_R(E(\widetilde{E}(M)))=\text{End}_R(E(M))$. From \cite[Theorem 2.8]{m}, we need to show that $f\widetilde{E}(M)\leq \widetilde{E}(M)$ for all  idempotents $f^2=f\in T$: Assume that $f\widetilde{E}(M)\nleq \widetilde{E}(M)$ for some idempotent $f^2=f\in T$. Then there exists  $x\in \widetilde{E}(M)$ such that $f(x)\notin\widetilde{E}(M)$. Thus, there exists $g\in T$ such that $g M=0$ and $g f(x)\neq 0$. Since $g f(x)\in E(M)$, there exists $r\in R$ such that $0\neq gf(xr)\in \widetilde{E}(M)$.
Thus, as $M\leq^\text{den}\widetilde{E}(M)$ and $xr\in \widetilde{E}(M)$, there exists $s\in R$ such that $0\neq g f(xrs)$ and $xrs\in M$. Note that $fM\leq M$ for all idempotents $f^2=f\in T$ because $M$ is quasi-continuous. However, $0\neq gf(xrs)\in g fM\leq gM=0$, a contradiction. Therefore $\widetilde{E}(M)$ is a quasi-continuous module.

(ii) The proof is similar to that of part (i) by using \cite[Corollary 1.14]{m}.
\end{proof}

\begin{rem}[{\cite[Theorem 5.3]{abt}}]\label{qiemqi2} The rational hull of every extending module is an extending module.
\end{rem}

Note that if $M$ is an injective module then $M=\widetilde{E}(M)$ (see \cite[Examples 8.18(1)]{l}).
The next examples exhibit that the converses of Theorem \ref{qiemqi} and Remark \ref{qiemqi2} do not hold true.

\begin{exam}
(i) Consider $\mathbb{Z}$ as a $\mathbb{Z}$-module.
Then $\widetilde{E}(\mathbb{Z})=\mathbb{Q}$ is (quasi-)injective, while $\mathbb{Z}$ is not quasi-injective.

(ii)(\cite[Example 2.9]{m}) Consider a ring $R=\left(\begin{smallmatrix} F& F  \\ 0 &  F\end{smallmatrix}\right)$ where $F$ is a field.
Then $\widetilde{E}(R_R)=\left(\begin{smallmatrix} F& F  \\ F &  F\end{smallmatrix}\right)$ is injective (hence, quasi-continuous), while $R_R$ is not quasi-continuous.

(iii) Consider $\mathbb{Z}^{(\mathbb{N})}$ as a $\mathbb{Z}$-module.
Then $\widetilde{E}(\mathbb{Z}^{(\mathbb{N})})=\mathbb{Q}^{(\mathbb{N})}$ is injective (hence, extending), while $\mathbb{Z}^{(\mathbb{N})}$ is not extending.
\end{exam}

\begin{cor} The maximal right ring of quotients of a quasi-continuous ring is also a quasi-continuous ring.
\end{cor}

\begin{rem}[{\cite[Exercises 13.8]{l}}]\label{qiemqi3} The maximal right ring of quotients of a simple (resp., prime, semiprime) ring is also a simple (resp., prime, semiprime) ring. 
\end{rem}

\begin{quest} Is the rational hull of a continuous module always a continuous module?
\end{quest}

\section{Direct sum of rational hulls of modules}

As we know, the injective hull of the direct sum of two modules is the direct sum of the injective hulls of each module without any condition. However, the rational hull case is different from the injective hull case. In this section, we discuss the condition for the rational hull of the direct sum of two modules to be the direct sum of the rational hulls of those modules.
The next example shows that the rational hull of the direct sum of two modules is not the direct sum of the rational hulls of each module, in general. 

\begin{exam}\label{notds} Consider $M=\mathbb{Z}\oplus\mathbb{Z}_p$ as a $\mathbb{Z}$-module where $p$ is prime.
Then $\widetilde{E}(\mathbb{Z})=\mathbb{Q}$ and $\widetilde{E}(\mathbb{Z}_p)=\mathbb{Z}_p$.
However, by \cite[Example 8.21]{l} $\widetilde{E}(M)=\mathbb{Z}_{(p)}\oplus\mathbb{Z}_p\neq \mathbb{Q}\oplus\mathbb{Z}_p$
where $\mathbb{Z}_{(p)}=\{\frac{m}{n}\in\mathbb{Q}\,|\,{m, n}\in\mathbb{Z}, (n, p)=1\}$.
Hence $M$ is not a dense submodule of $\mathbb{Q}\oplus\mathbb{Z}_p$: For $(\frac{1}{p}, \overline{0})$ and $0\neq(0, \overline{1})\in\mathbb{Q}\oplus\mathbb{Z}_p$, there is no $n\in \mathbb{Z}$ such that $n(\frac{1}{p}, \overline{0})\in\mathbb{Z}\oplus\mathbb{Z}_p$ and $n(0, \overline{1})\neq0$.
\end{exam}

\begin{prop}\label{noefinite} Let $M=\bigoplus_{k\in \Lambda} M_k$ where $M_k$ be a right $R$-module and $\Lambda$ is any index set.
If either $R$ is right noetherian or $|\Lambda|$ is finite, then $\widetilde{E}(M)\leq\bigoplus_{k\in \Lambda}\widetilde{E}(M_k)$.
\end{prop}
\begin{proof} Suppose $0\neq m\in\widetilde{E}(M)$.
Since $\widetilde{E}(M)\subseteq E(M)=\oplus_{k\in \Lambda}E(M_k)$ because $R$ is right noetherian or $|\Lambda|$ is finite,
there exists $\ell\in\mathbb{N}$ such that $m\in\oplus_{i=1}^\ell E(M_i)$.
Thus, $m=(m_1,\dots, m_\ell)$ where $m_i\in E(M_i)$.
Since $(0, \dots, 0, y_i, 0, \dots, 0)\cdot m^{-1}M\neq 0$ for all $0\neq y_i\in E(M_i)$  and $m^{-1}M=m_1^{-1}M_1\cap \cdots \cap m_\ell^{-1}M_\ell$, $y_i\cdot m_i^{-1}M_i\neq 0$  for all $0\neq y_i\in E(M_i)$.
Thus, $m_i\in \widetilde{E}(M_i)$ for all $1\leq i\leq \ell$ from Proposition \ref{rationalhull}.
So, $m=(m_1,\dots, m_\ell)\in\oplus_{i=1}^\ell\widetilde{E}(M_i)\subseteq \oplus_{k\in \Lambda}\widetilde{E}(M_k)$.
Therefore $\widetilde{E}(M)\leq \oplus_{k\in \Lambda}\widetilde{E}(M_k)$.
\end{proof}

\begin{rem} Example \ref{notds} illustrates Proposition \ref{noefinite} because $R=\mathbb{Z}$ is a noetherian ring, that is, $\widetilde{E}(\mathbb{Z}\oplus\mathbb{Z}_p)=\mathbb{Z}_{(p)}\oplus\mathbb{Z}_p\lneq \mathbb{Q}\oplus\mathbb{Z}_p=\widetilde{E}(\mathbb{Z})\oplus\widetilde{E}(\mathbb{Z}_p)$. However, Example \ref{sumexam} shows that  the condition ``either $R$ is right noetherian or $|\Lambda|$ is finite" is not superfluous because
$\widetilde{E}(\oplus_{k\in\Lambda}\mathbb{Z}_2)=\prod_{k\in\Lambda}\mathbb{Z}_2\gneq \oplus_{k\in\Lambda}\mathbb{Z}_2=\oplus_{k\in\Lambda}\widetilde{E}(\mathbb{Z}_2)$ with a non-noetherian ring $R=\langle \oplus_{k\in\Lambda}\mathbb{Z}_2, 1\rangle$.
\end{rem}

To get the reverse inclusion of Proposition \ref{noefinite}, first we provide the properties of the relatively dense property. 
\begin{lem}\label{ndense2} Let $N\leq M$ and $K_i$ be right $R$-modules for all $i\in \Lambda$. Then $N$ is $K_i$-dense in $M$ for all $i\in \Lambda$ if and only if $N$ is $\bigoplus_{i\in \Lambda}K_i$-dense in $M$ if and only if $N$ is $\bigoplus_{i\in \Lambda}\widetilde{E}(K_i)$-dense in $M$.
\end{lem}
\begin{proof} Let $P$ be any submodule such that $N\leq P\leq M$.
Since $N$ is $K_i$-dense in $M$, $\text{Hom}_R(P/N, K_i)=0$ for all $i\in \Lambda$ from Proposition \ref{ndense}. Consider the sequence $0\rightarrow \oplus_{i\in \Lambda}K_i\rightarrow\prod_{i\in \Lambda}K_i$. Then we have  $0\rightarrow \text{Hom}_R(P/N, \oplus_{i\in \Lambda}K_i)\rightarrow\text{Hom}_R(P/N,\prod_{i\in \Lambda}K_i)\cong \prod_{i\in \Lambda}(\text{Hom}_R(P/N, K_i)=0$.
Thus $\text{Hom}_R(P/N, \oplus_{i\in \Lambda}K_i)=0$. Therefore $N$ is $\oplus_{i\in \Lambda}K_i$-dense in $M$ from Proposition \ref{ndense}.
Conversely, since $\text{Hom}_R(P/N, \oplus_{i\in \Lambda}K_i)=0$, $\text{Hom}_R(P/N, K_i)=0$ for each $i\in \Lambda$.
Hence $N$ is $K_i$-dense in $M$ for all $i\in \Lambda$.\\ For the second equivalence, since $E(\widetilde{E}(K_i))=E(K_i)$, from Proposition \ref{ndense} it is easy to see that $N$ is $K_i$-dense in $M$ if and only if $N$ is $\widetilde{E}(K_i)$-dense in $M$, for all $i\in \Lambda$. After using the first equivalence, we have the second equivalence.
\end{proof}

Using Lemma \ref{ndense2}, we obtain a characterization for $\oplus_{k\in \Lambda} N_k$ to be a dense submodule of $\oplus_{k\in \Lambda} M_k$ where $N_i$ is a submodule of $M_i$ for each $i\in \Lambda$.
\begin{prop}\label{ndense5} Let $N_i\leq M_i$ be right $R$-modules for all $i\in \Lambda$ where $\Lambda$ is any index set.
Let $N=\bigoplus_{k\in \Lambda} N_k$ and $M=\bigoplus_{k\in \Lambda} M_k$. Then $N\leq^\emph{den}M$ if and only if $N_i$ is $M_j$-dense in $M_i$ for all $i, j\in \Lambda$.
\end{prop}
\begin{proof}
Suppose $N\leq^\text{den}M$. Then $N$ is $M$-dense in $M$ by the definition. From Lemma \ref{ndense2}  $N$ is $M_j$-dense in $M$ for all $j\in \Lambda$.
Let $x_i\in M_i$ and $0\neq y_j\in M_j$ be arbitrary  for each $i, j\in \Lambda$. Since $(0,\dots,0,x_i,0,\dots) \in M$ and $0\neq y_j\in M_j$, there exists $r\in R$ such that $(0,\dots,0,x_i,0,\dots)r \in N$ and $y_jr\neq0$. Since $x_ir\in N_i$ and $y_jr\neq0$, $N_i$ is $M_j$-dense in $M_i$ for all $i, j\in \Lambda$.

Conversely, suppose $N_i$ is $M_j$-{dense} in $M_i$ for all $i, j \in \Lambda$.
From Lemma \ref{ndense2}, $N_i$ is $\oplus_{k\in \Lambda} M_k$-{dense} in $M_i$ for all $i \in \Lambda$.
Let $x\in M$ and $0\neq y\in M$ be arbitrary. Then  there exists $\ell\in\mathbb{N}$ such that $x=(x_1, \dots, x_\ell)\in\oplus_{k=1}^\ell M_k\leq M$. Since $N_1$ is $M$-dense in $M_1$, there exists $r_1\in R$ such that $x_1r_1\in N_1$ and $0\neq yr_1\in M$. Also, since $N_2$ is $M$-dense in $M_2$, there exists $r_2\in R$ such that $x_2r_1r_2\in N_2$ and $0\neq yr_1r_2\in M$. 
By the similar processing, we have $r=r_1r_2\cdots r_\ell\in R$ such that $xr\in\oplus_{k=1}^\ell N_k\leq N$ and $yr\neq0$.
Therefore $N\leq^\text{den}M$.
\end{proof}

From Propositions \ref{noefinite} and \ref{ndense5}, we have a characterization for the rational hull of the direct sum of modules to be the direct sum of the rational hulls of each module.
\begin{thm}\label{directsum} Let $M=\bigoplus_{k\in \Lambda} M_k$ where $M_k$ is a right $R$-module and $\Lambda$ is any index set.
If either $R$ is right noetherian or $|\Lambda|$ is finite, then  $\widetilde{E}(M)=\bigoplus_{k\in \Lambda} \widetilde{E}(M_k)$ if and only if $M_i$ is $M_j$-{dense} in $\widetilde{E}(M_i)$ for all $i, j \in \Lambda$.
\end{thm}
\begin{proof} Suppose $\widetilde{E}(M)=\oplus_{k\in \Lambda} \widetilde{E}(M_k)$.
Since $M\leq^\text{den}\oplus_{k\in \Lambda} \widetilde{E}(M_k)$,
from Proposition \ref{ndense5} $M_i$ is $\widetilde{E}(M_j)$-dense in $\widetilde{E}(M_i)$ for all $i, j \in \Lambda$.
Thus, $M_i$ is $M_j$-dense in $\widetilde{E}(M_i)$ for all $i, j \in \Lambda$ from Lemma \ref{ndense2}.

Conversely, suppose $M_i$ is $M_j$-{dense} in $\widetilde{E}(M_i)$ for all $i, j \in \Lambda$.
Then $M_i$ is $\widetilde{E}(M_j)$-dense in $\widetilde{E}(M_i)$ for all $i, j \in \Lambda$ from Lemma \ref{ndense2}. 
Thus, from Proposition \ref{ndense5} $M\leq^{\text{den}}\oplus_{k\in \Lambda} \widetilde{E}(M_k)$. 
Hence $\oplus_{k\in \Lambda} \widetilde{E}(M_k)\leq\widetilde{E}(M)$ from Proposition \ref{pro137}. 
Also, from Proposition \ref{noefinite} $\widetilde{E}(M)\leq\oplus_{k\in \Lambda} \widetilde{E}(M_k)$.
Therefore $\widetilde{E}(M)=\oplus_{k\in \Lambda} \widetilde{E}(M_k)$.
\end{proof}

The next examples show that the condition ``$R$ is right noetherian or $|\Lambda|$ is finite" in Theorem \ref{directsum} is not superfluous.
\begin{exam}\label{sumexam}(i) Let $R=\langle \oplus_{k\in\Lambda}\mathbb{Z}_2, 1\rangle$ and $M=\oplus_{k\in\Lambda}M_k$ where $M_k=\mathbb{Z}_2$. 
Note that $R$ is not noetherian.
Since $\mathbb{Z}_2$ is an injective $R$-module, $\widetilde{E}(\mathbb{Z}_2)=\mathbb{Z}_2$. Thus $M_i$ is $M_j$-{dense} in $\widetilde{E}(M_i)$ for all $i, j \in \Lambda$. However, $\widetilde{E}(\oplus_{k\in\Lambda}\mathbb{Z}_2)=\prod_{k\in\Lambda}\mathbb{Z}_2\gneq \oplus_{k\in\Lambda}\mathbb{Z}_2=\oplus_{k\in\Lambda}\widetilde{E}(\mathbb{Z}_2)$.

(ii) Let $R=\{(a_k)\in \prod_{k\in\Lambda}\mathbb{Z}\,|\, a_k~ \text{is eventually constant}\}$ and $M=\oplus_{k\in\Lambda}\mathbb{Z}$.
Note that $R$ is not noetherian.
Then $\widetilde{E}(\mathbb{Z})=\mathbb{Q}$ and $\mathbb{Z}$ is $\mathbb{Z}$-{dense} in $\widetilde{E}(\mathbb{Z})$. However, $\widetilde{E}(\oplus_{k\in\Lambda}\mathbb{Z})=\prod_{k\in\Lambda}\mathbb{Q}\gneq \oplus_{k\in\Lambda}\mathbb{Q}=\oplus_{k\in\Lambda}\widetilde{E}(\mathbb{Z})$.
\end{exam}

The next example illustrates  Theorem \ref{directsum}.
\begin{exam}\label{sumexam2} Consider $M=\mathbb{Z}\oplus\mathbb{Z}_p$ as a $\mathbb{Z}$-module where $p$ is prime.
Then $\mathbb{Z}_p$ is $\mathbb{Z}$-dense in $\widetilde{E}(\mathbb{Z}_p)=\mathbb{Z}_p$, but $\mathbb{Z}$ is not $\mathbb{Z}_p$-dense in $\mathbb{Q}$ because for $\frac{1}{p}\in \mathbb{Q}, \overline{1}\in \mathbb{Z}_p$, there is no element $t\in \mathbb{Z}$ such that $t\frac{1}{p}\in\mathbb{Z}$ and $t\overline{1}\neq 0$.
Thus, from Theorem \ref{directsum} $\widetilde{E}(M)=\mathbb{Z}_{(p)}\oplus\mathbb{Z}_p\lneq \mathbb{Q}\oplus\mathbb{Z}_p=\widetilde{E}(\mathbb{Z})\oplus\widetilde{E}(\mathbb{Z}_p)$. (See Example \ref{notds} for details.)
\end{exam}

\begin{cor}\label{noefinite2} Let $M$ be a right $R$-module.
If either $R$ is right noetherian or $\Lambda$ is a finite index set, then  $\widetilde{E}(M^{(\Lambda)})=(\widetilde{E}(M))^{(\Lambda)}$.
\end{cor}

\begin{cor}\label{cordirectsum} Let $\{M_k\}_{k\in\Lambda}$ be a class of rationally complete right $R$-module for any index set $\Lambda$.
If either $R$ is right noetherian or $|\Lambda|$ is finite, then $M=\bigoplus_{k\in \Lambda} M_k$ is rationally complete.
\end{cor}
\begin{proof} Since $\widetilde{E}(M_i)=M_i$, $M_i$ is $M_j$-{dense} in $\widetilde{E}(M_i)$ for all $i, j \in \Lambda$.
From Theorem \ref{directsum},  $\widetilde{E}(M)=\oplus_{k\in \Lambda} \widetilde{E}(M_k)=\oplus_{k\in \Lambda} M_k=M$. Therefore $\oplus_{k\in \Lambda} M_k$ is rationally complete.
\end{proof}

\begin{prop}[{\cite[Proposition 1.9]{sto}}]
	Let $\{S_i\}_{i\in \Lambda}$ be a set of nonisomorphic simple modules, representing all singular simple modules. Then every module containing the module $P=\bigoplus_{i\in\Lambda} S_i$ is rationally complete.
\end{prop}

\section{The endomorphism ring of a module over a right ring of quotients of a ring}

In this section, we obtain some condition to be $\text{End}_R(M)=\text{End}_H(M)$ where $H$ is a right ring of quotients of a ring $R$.
Recall that an extension ring $H$ of a ring $R$ is called a \emph{right ring of quotients} of
$R$ if for any two elements $x\neq0$ and $y$ of $H$, there exists an element $r\in R$ such that $xr\neq0$ and $yr\in R$. 

\begin{thm}\label{subprojectivemaximal}
Let $M$ be a right $H$-module where $H$ is a right ring of quotients of a ring $R$.
If $R$ is $M_R$-dense in $H_R$ then $\emph{End}_R(M)=\emph{End}_H(M)$.
\end{thm}
\begin{proof} Since $\text{End}_H(M) \subseteq \text{End}_R(M)$, it suffices to show that $\text{End}_R(M) \subseteq \text{End}_H(M):$  
 Let $\varphi \in \text{End}_R(M)$ be arbitrary. Assume that $\varphi \notin \text{End}_H(M)$. Then there exist $m \in M, t \in H$ such that $\varphi(mt)-\varphi(m)t \neq 0.$ Since $R$ is $M_R$-dense in $H_R$, there exists $r \in R$ such that $(\varphi(mt)-\varphi(m)t)r \neq 0$ and $tr \in R$. Hence $0\neq (\varphi(mt)-\varphi(m)t)r=\varphi(mt)r-\varphi(m)(tr)=\varphi(mtr)-\varphi(mtr)=0$, a contradiction. Therefore $\text{End}_R(M)=\text{End}_H(M)$. 
\end{proof}

\begin{rem}(i) $R$ is always $E(R)$-dense in $H_R$ where $H$ is a right ring of quotients of $R$. For, let $x\in H_R$ and $0\neq y\in E(R)$.
Since $H\leq^\text{ess} E(R)_R$, there exists $s\in R$ such that $0\neq ys\in H$. Also, $xs\in H$.
Since $R\leq^\text{den} H_R$, there exists $t\in R$ such that $xst\in R$ and  $0\neq yst$. Therefore $R$ is $E(R)$-dense in $H_R$.\\
(ii) If $M$ is a nonsingular $R$-module, then $R$ is $M_R$-dense in $H_R$:
For, let $0\neq m\in M$ and $t\in H$ be arbitrary. Take $t^{-1}R=\{r\in R\,|\,tr\in R\}$ a right ideal of $R$. Note that $t^{-1}R\leq^{\textrm{ess}}R_R$.
Since $t^{-1}R \nleq^\text{ess} \mathbf{r}_R(m)$, there exists $r\in t^{-1}R$ and $r\notin\mathbf{r}_R(m)$. Thus, $tr\in R$ and $mr\neq 0$. Therefore $R$ is $M_R$-dense in $H_R$.\\
(iii)	If $M$ is a submodule of a projective right $H$-module, then $R$ is $M_R$-dense in $H_R$. For, let $P$ be a projective right $R$-module including $M$, that is, $M \leq P$ where $P \leq^{\oplus } H^{(\Lambda)}$ with  some index set $\Lambda$. Then there is a right $R$-module $K \le E(P)$ such that $E(P)=E(M) \oplus K.$ Since $R \le^\text{den} H_R$, we get that $R$ is $H^{(\Lambda)}$-dense in $H_R$ from Lemma \ref{ndense2}. Hence $R$ is $P$-dense in $H_R$. Thus $\text{Hom}_R(H/R,E(P))=0$ from Proposition \ref{ndense}. Since $\text{Hom}_R(H/R,E(P)) \cong \text{Hom}_R(H/R,E(M)) \oplus \text{Hom}_R(H/R,K)$, we obtain $\text{Hom}_R(H/R,E(M))=0$. It follows that $R$ is $M_R$-dense in $H_R.$
\end{rem}

\begin{cor}\label{subprojectivemaximal2}
Let $M$ be a projective right $H$-module where $H$ is a right ring of quotients of $R$. Then $\emph{End}_R(M)=\emph{End}_H(M)$.
\end{cor}
The next example illustrates Corollary \ref{subprojectivemaximal2}.
\begin{exam} Let $Q=\prod_{n=1}^\infty \mathbb{Z}_2$ and  $R=\{(a_n) \in
Q~|~a_n$ is eventually constant$\}$.
Then $Q$ is a maximal right ring of quotients of $R$. Hence from Theorem \ref{subprojectivemaximal},
End$_R(Q^{(\Lambda)})$= End$_Q(Q^{(\Lambda)})=\mathsf{CFM}_\Lambda(Q)$.
\end{exam}

\begin{thm}\label{maximalrational} Let $M$ be a finitely generated free $R$-module with $S=\emph{End}_R(M)$.
If either $R$ is right noetherian or $\Lambda$ is any finite index set,
then $\emph{End}_R\left(\widetilde{E}(M^{(\Lambda)})\right)=\mathsf{CFM}_{\Lambda}\left(Q(S)\right)$.
\end{thm}
\begin{proof} Let $M=R^{(n)}$ for some $n\in\mathbb{N}$.
From Corollary \ref{noefinite2}, $\widetilde{E}(R^{(n)})=\widetilde{E}(R)^{(n)}=Q(R)^{(n)}$ as $\widetilde{E}(R)=Q(R)$.
Hence $\text{End}_R\left(\widetilde{E}(M^{(\Lambda)})\right)=\text{End}_R\left(\widetilde{E}(M)^{(\Lambda)}\right)=\text{End}_R\left((Q(R)^{(n)})^{(\Lambda)}\right)=\text{End}_{Q(R)}\left((Q(R)^{(n)})^{(\Lambda)}\right)=\mathsf{CFM}_{\Lambda}\left(\text{End}_{Q(R)}(Q(R)^{(n)})\right)=\mathsf{CFM}_{\Lambda}\left(\mathsf{Mat}_n(Q(R))\right)$  from Theorem \ref{subprojectivemaximal}.
Therefore $\text{End}_R\left(\widetilde{E}(M^{(\Lambda)})\right)=\mathsf{CFM}_{\Lambda}\left(Q(\text{End}_R(M))\right)$ because $\mathsf{Mat}_n(Q(R))=Q(\mathsf{Mat}_n(R))$ by \cite[2.3]{ut2} and $\text{End}_R(M)=\mathsf{Mat}_n(R)$.
\end{proof}

The next result is generalized from \cite[2.3]{ut2}.
\begin{cor} Let $M$ be a finitely generated free $R$-module.
Then $Q(\emph{End}_R(M))=\emph{End}_R(\widetilde{E}(M))$.
\end{cor}

The following example shows that the above result can not be extended to {\it flat} modules. This example also shows that $R$ is not $M_R$-dense in $Q_R$ where $Q$ is a right ring of quotients of $R$.
\begin{exam}
Let $Q=\prod_{n=1}^\infty \mathbb{Z}_2$, $R=\{(a_n) \in
Q~|~a_n$ is eventually constant$\}$, and  $I=\{ (a_n) \in
Q~|~a_n=0~\textrm{eventually}\}$.
Note that $Q=Q(R)$.
Let $M=Q/I$, which is a flat $Q$-module but not projective.
We claim that End$_Q(M) \subsetneq$ End$_R(M)$.
Indeed, define $f: M\to M$ via
$$f[(a_1, a_2, \dots, a_n, a_{n+1}, \dots)+I] = (a_1, 0, a_2, 0, \dots, a_n, 0, a_{n+1}, 0, \dots)+I,$$
for any $\overline{a} = a+I=(a_1, a_2, \dots, a_n, a_{n+1}, \dots)+I \in M$.
It is easy to see that $f(\overline{a}+\overline{b})$ = $f(\overline{a})$ + $f(\overline{b})$
for any $\overline{a}$, $\overline{b} \in M$.
Meanwhile, for any $r = (r_1, r_2, \dots, r_n, r_{n+1}, \dots)\in R$,
we have
$$(a+I)r = ar + I
= \left\{\begin{array}{ll}
         (0, \dots, 0, a_n, a_{n+1}, \dots)+I, & \mbox{if } r_n = r_{n+1} = \dots = 1; \\
         (0, \dots, 0, 0, 0, \dots)+I,         & \mbox{if } r_n = r_{n+1} = \dots = 0.
         \end{array}
  \right.$$
Note that $a+I=(0, 0, \dots, 0, a_n, a_{n+1}, \dots)+I$ for some $n\in \mathbb{N}$.
One can easily see that $f[(a+I)r]$ = $[f(a+I)]r$ for all $a\in Q$, $r\in R$.
This shows $f\in$ End$_R(M)$.
However, for $q = (0, 1, 0, 1, \dots) = q^2 \in Q$,
we have $[f(q+I)]q = 0+I$ while $f[(q+I)q] = f(q+I) \neq 0+I$.
This means $f\not\in$ End$_Q(M)$.
Thus, End$_Q(M) \subsetneq$ End$_R(M)$. Note that $R$ is not $M_R$-dense in $Q$.
 For, let $q\in Q\setminus R$ and $m=1+I\in M$. Since $(1+I)r=0+I$ for all $r\in R\setminus {1}$, it has to be $r=1$ to get $mr\neq 0+I$.
However, $qr \notin R$. 
\end{exam}

Recall that a module $M$ is said to be \emph{polyform} if every essential submodule of $M$ is a dense submodule.
\begin{lem}\label{perknmhr} 
 A module $M$ is polyform iff $\widetilde{E}(M)$ is a polyform quasi-injective module.
\end{lem}
\begin{proof}
Let $X$ be essential in $\widetilde{E}(M)$.
Then $X\cap M\leq^\text{ess} M$. Hence $X\cap M$ is a dense submodule of $M$ because $M$ is polyform.
Since $X\cap M\leq^\text{den} M\leq^\text{den}\widetilde{E}(M)$, $X\cap M\leq^\text{den}\widetilde{E}(M)$.
Thus $X$ is a dense submodule of $\widetilde{E}(M)$ from Proposition \ref{pro87}(ii).
Therefore  $\widetilde{E}(M)$ is a polyform module.
In addition, $\widehat{M}$ is also a polyform module from \cite[11.1]{wr}.
Since $M\leq^\text{ess}\widehat{M}$, $M\leq^\text{den}\widehat{M}$. 
Thus $\widetilde{E}(M)=\widetilde{E}(\widehat{M})$.
Since the rational hull of a quasi-injective module is also quasi-injective from Theorem \ref{qiemqi},
$\widetilde{E}(M)$ is a quasi-injective module.
Therefore $\widetilde{E}(M)$ is a polyform quasi-injective module.

Conversely, let $N$ be any essential submodule of $M$. Then $N$ is also essential in $\widetilde{E}(M)$.
Hence $N$ is a dense submodule of $\widetilde{E}(M)$ as $\widetilde{E}(M)$ is  polyform.
So $N$ is a dense submodule of $M$. Therefore $M$ is polyform.
\end{proof}

We show from Theorem \ref{relationship} that  there is a canonical embedding of the ring $\text{End}_R(M)$ into the ring $\text{End}_R(\widetilde{E}(M))$. Next, we obtain a condition when $\text{End}_R(M)$ and $\text{End}_R(\widetilde{E}(M))$ are isomorphic. It is a generalization of \cite[Exercises 7.32]{l}.
\begin{prop}\label{qiiso} If $M$ is a quasi-injective module
then $\emph{End}_R(M)\stackrel{\Omega}{\cong}\emph{End}_R(\widetilde{E}(M))$. In particular, if $M$ is a polyform module, then the converse holds true.
\end{prop}
\begin{proof} In the proof of Theorem \ref{relationship}, we only need to show that $\Omega: \text{End}_R(M) \rightarrow \text{End}_R(\widetilde{E}(M))$ given by $\Omega(\varphi)=\widetilde{\varphi}$, is surjective: Let $\psi\in\text{End}_R(\widetilde{E}(M))$ be arbitrary. Then there exists $\widehat{\psi}\in\text{End}_R(E(M))$ such that $\widehat{\psi}|_{\widetilde{E}(M)}=\psi$. Since $\widehat{\psi}M\leq M$ as $M$ is quasi-injective, $\widehat{\psi}|_M=\psi|_M\in\text{End}_R(M)$. Thus, $\Omega(\psi|_M)=\psi$, which shows that $\Omega$ is surjective.

In addition, suppose that $M$ is a polyform module. Then from Lemma \ref{perknmhr}, $\widetilde{E}(M)$ is quasi-injective.
Thus, for any $\vartheta\in\text{End}_R(E(M))$, $\vartheta\widetilde{E}(M)\leq\widetilde{E}(M)$.
Since $\vartheta|_{\widetilde{E}(M)}\in \text{End}_R(\widetilde{E}(M))$ and $\text{End}_R(M)\stackrel{\Omega}{\cong}\text{End}_R(\widetilde{E}(M))$,
there exists $\varphi\in\text{End}_R(M)$ such that $\Omega(\varphi)=\vartheta|_{\widetilde{E}(M)}$.
Also by Theorem \ref{relationship}, $\vartheta|_M=\varphi$.
Thus, $\vartheta M=\varphi M\leq M$.
Therefore $M$ is a quasi-injective module.
\end{proof}

\begin{cor} If $M$ is a quasi-injective module,
then $Q(\emph{End}_R(M))\cong\emph{End}_R(\widetilde{E}(M))$.
\end{cor}
\begin{proof} Since $M$ is a quasi-injective module $\text{End}_R(M)$ is a right self-injective ring.
So, $Q(\text{End}_R(M))=\text{End}_R(M)$. Thus, $Q(\text{End}_R(M))\cong\text{End}_R(\widetilde{E}(M))$ by Proposition \ref{qiiso}.
\end{proof}

Remark that if $M$ is a quasi-injective module then $\widetilde{E}(M)$ is a quasi-injective module from Theorem \ref{qiemqi} and  $\text{End}_R(M)\cong\text{End}_R(\widetilde{E}(M))$ from Proposition \ref{qiiso}. However, the next example shows that there exists a quasi-injective module $M$ such that $M\neq\widetilde{E}(M)$.

\begin{exam} Let  $R=\left(\begin{smallmatrix}F&F\\0&F\end{smallmatrix}\right)$ and $M=\left(\begin{smallmatrix}0&0\\0&F\end{smallmatrix}\right)$ where $F$ is a field. Then $M$ is a quasi-injective $R$-module. However, $\widetilde{E}(M)=E(M)=\left(\begin{smallmatrix}0&0\\F&F\end{smallmatrix}\right)$ because $M$ is nonsingular.
Thus $M$ is a quasi-injective $R$-module such that $M\lneq\widetilde{E}(M)$ and $\text{End}_R(M)\cong\left(\begin{smallmatrix}0&0\\0&F\end{smallmatrix}\right)\cong\text{End}_R(\widetilde{E}(M))$.
\end{exam}

Because $\widetilde{E}(M)=E(M)$ for a right nonsingular module $M$, we have the following well-known results as a consequence of Proposition \ref{qiiso}.
\begin{cor}[{\cite[Exercises 7.32]{l}}] For any nonsingular module $M$, the following statements hold true:
\begin{enumerate}
\item[(i)] there is a canonical embedding $\Omega$ of the ring $\emph{End}_R(M)$ into the ring $\emph{End}_R(E(M))$.
\item[(ii)] $M$ is a quasi-injective $R$-module if and only if $\Omega$ is an isomorphism.
\end{enumerate}
\end{cor}

\begin{cor} Let $M$ be a right $H$-module where $H$ is a right ring of quotients of a ring $R$. The following statements hold true:
\begin{enumerate}
\item[(i)] If $M$ is a nonsingular $R$-module then $\emph{End}_R(M)=\emph{End}_H(M)$.
\item[(ii)] If $M$ is a submodule of a projective $H$-module, then $\emph{End}_R(M)=\emph{End}_H(M)$.
\item[(iii)] If $M$ is a nonsingular quasi-injective $R$-module then $\emph{End}_R(M)\cong\emph{End}_R(E(M))$ and $\emph{End}_H(M)\cong\emph{End}_H(E(M))$.
\item[(iv)] If $M$ is a quasi-injective $R$-module and is a submodule of a projective $H$-module then $\emph{End}_R(M)\cong\emph{End}_R(\widetilde{E}(M))$ and $\emph{End}_H(M)\cong\emph{End}_H(\widetilde{E}(M))$.
\end{enumerate}
\end{cor}

\end{document}